\documentclass[UTF8,12pt]{article}
\usepackage{amsmath}
\usepackage{amsthm}
\usepackage[colorlinks, linkcolor=blue, anchorcolor=blue, citecolor=blue]{hyperref}
\usepackage{graphicx}
\usepackage{textcomp}
\usepackage{listings}
\usepackage{amsfonts}
\usepackage[textwidth=14.5cm]{geometry}
\usepackage{blindtext}
\usepackage{hyperref}
\newtheorem{thm}{Theorem}[section]
\newtheorem{cor}[thm]{Corollary}
\newtheorem{lemma}[thm]{Lemma}
\theoremstyle{definition}

\title{Gradient estimates for $(p,V)$-harmonic functions on Riemannian manifolds}
\author{Yuxin Dong, Hezi Lin, Weihao Zheng}
\date{ }
\begin{document}
\maketitle
\begin{abstract}
In this paper, we study $(p,V)$-harmonic functions on complete Riemannian manifolds using the Moser iteration method. A volume comparison theorem and a Sobolev embedding theorem are established under the Bakry-$\rm\acute{E}$mery curvature condition. Moreover, we obtain an explicit global gradient estimate for positive entire $(p,V)$-harmonic functions.
\end{abstract}

\section{Introduction}
\qquad Gradient estimate is one of the most important techniques in geometric analysis, attracting the attention of many mathematicians due to their wide-ranging applications in mathematics, physics, and numerous other fields. It was developed by Yau \cite{YauST75} while studying the Liouville theorem for harmonic functions on complete Riemannian manifolds. His argument was further localized in a joint paper with Cheng \cite{ChengYau75}, which led to a gradient estimate for a broader class of elliptic equations. In 1986, Li and Yau \cite{LiYau86} introduced the renowned Li-Yau type gradient estimate, a parabolic gradient estimate for the solution to the parabolic Schr$\rm\ddot{o}$dinger equation on Riemannian manifolds with Ricci curvature bounded from below. 

\qquad The classical method for obtaining gradient estimates of solutions to elliptic or parabolic equations is the maximum principle. To establish local gradient estimates of solutions on complete manifolds by this method, one needs to apply the relevant second-order differential operator to an auxiliary function constructed by the solution and a cutoff function, where the cutoff function is related to the distance function. Therefore, when studying solutions of equations related to the Laplace operator, the Laplacian comparison theorem of the distance function is often used. For equations involving more complicated operators, such as the $p$-Laplace operator \cite{MaZhu21} or the exponential Laplace operator \cite{JiangMao24}, the Hessian comparison theorem of the distance function is required. Recall that the Hessian comparison theorem is established under the sectional curvature condition. Consequently, Kortschwar and Ni \cite{KotschwarNi09} gave a gradient estimate for $p$-harmonic functions on complete Riemannian manifolds under a sectional curvature condition.

\qquad In 2011, Wang and Zhang \cite{WangZhang11} introduced a Moser iteration technique to derive a Cheng-Yau type gradient estimate for $p$-harmonic functions. Their method only differentiates the distance function once, thus only the Ricci curvature condition is required. Specifically, they established the following gradient estimate for $p$-harmonic functions.
\begin{thm}\label{th1.1}
	{\rm (\cite{WangZhang11})} Let $(M^n,g)$ be a complete Riemannian manifold, and $o\in M$ be a fixed point. Assume the Ricci curvature satisfies $Ric\geq-(n-1)K$ on $B_o(2R)$, $K\geq0$. Suppose $u$ is a positive $p$-harmonic function, i.e., a positive solution to
	\[\Delta_pu=div(\lvert\nabla u\rvert^{p-2}\nabla u)=0.\]
	Then for any $x\in B_o(R)$, there exists a constant $C=C(n,p)$ such that
	\[\frac{\lvert\nabla u\rvert}{u}\leq C\frac{1+\sqrt{K}R}{R}.\]
\end{thm}

Clearly, a $2$-harmonic function is simply a harmonic function. We omit $p$ when $p=2$ in the rest of this article.

\qquad Let $V$ be a smooth vector field on $M$. A map $u:M\to X$ is called a $V$-harmonic map if it satisfies
\[\delta d(u)+du(V)=0.\]
The $V$-harmonic map was first introduced and studied in \cite{ChenJostQiu12,ChenJostWang15}, and it is a natural generalization of harmonic maps. Notable examples of $V$-harmonic maps include the Hermitian harmonic map \cite{JostYau93} and the Weyl harmonic map \cite{Kokarev09}.

\qquad Another generalization of harmonic maps and $p$-harmonic functions is the $p$-harmonic map. A map $u:M\to X$ is said to be a $p$-harmonic map for $p>1$ if it satisfies
\[\delta(\lvert du\rvert^{p-2}du)=0.\]
It is derived from the $p$-energy functional as its Euler–Lagrange equation. For a more in-depth discussion of the $p$-harmonic map, one can refer to \cite{WeiSW08}. Using the Moser iteration technique, Dong and Lin \cite{DongLin21} derived a gradient estimate and a Liouville theorem for the $p$-harmonic map in 2021.

\qquad The $(p,V)$-harmonic map is a combination of the $p$-harmonic map and the $V$-harmonic map, which satisfies
\[\delta(\lvert du\rvert^{p-2}du)+\lvert du\rvert^{p-2}du(V)=0.\]
The $(p,V)$-harmonic function is a special case when $X=\mathbb{R}$. In this paper, we will give a gradient estimate for a positive $(p,V)$-harmonic function, i.e., a positive solution to 
\begin{equation}\label{1.1}
	\Delta_{p,V}u=0,
\end{equation}
where $V$ is a smooth vector field.

\begin{thm}\label{th1.2}
	Let $(M^n,g)$ be a complete Riemannian manifold, and $o\in M$ be a fixed point. Assume the Bakry-$\acute{E}$mery curvature satisfies $Ric_V\geq-(n-1)K$ on $B_o(2R)$, $K\geq0$. Suppose $\lvert V\rvert\leq \theta$ in $B_o(2R)$, and $u$ is a positive solution to (\ref{1.1}) in $B_o(2R)$, then for any $x\in B_o(R)$, there exists a constant $C=C(n,p)$ such that
	\[\frac{\lvert\nabla u\rvert}{u}\leq C\frac{1+(\sqrt{K}+\theta)R}{R}.\]
\end{thm}
\qquad It can be seen that our result recaptures that in \cite{WangZhang11} by letting $V=0$. When $V=-\nabla f$ is a gradient vector field, the $(p,V)$-harmonic function becomes the $(p,f)$-harmonic function, which satisfies
\begin{equation*}
	\begin{aligned}
		\Delta_{p,f}u= e^fdiv(e^{-f}\lvert\nabla u\rvert^{p-2}\nabla u)=0. 
	\end{aligned}
\end{equation*}
The $f$-Laplace operator $\Delta_f=\Delta-\langle \nabla\ ,\nabla f\rangle$ is the self-adjoint operator of the weighted volume $e^{-f}dvol_g$, thus the operator $\Delta_{p,f}$  can be expressed in divergence form. However, the $(p,V)$-Laplace operator can not be transformed into divergence form when $V$ is an arbitrary vector field. This provides the main obstacle of the integration-based Moser iteration method. The method we use to deal with it is inspired by \cite{SongWuZhu23,SongWuZhu25}.

\qquad From Theorem \ref{th1.2} we can obtain the following Harnack inequality as an application.
\begin{cor}\label{cor1.3}
	Let $(M^n,g)$ be a complete Riemannian manifold, and $o\in M$ be a fixed point. Assume the Bakry-$\acute{E}$mery curvature satisfies $Ric_V\geq-(n-1)K$ on $B_o(2R)$, $K\geq0$. Suppose $\lvert V\rvert\leq \theta$ in $B_o(2R)$, and $u$ is a positive solution to (\ref{1.1}) in $B_o(2R)$. Then for any $x,y\in B_o(R)$, there exists a constant $C=C(n,p)$ such that
	\[u(x)\leq u(y)exp\bigl(C(1+(\sqrt{K}+\theta)R)\bigr).\]
\end{cor}
\qquad A global estimate can be got from Theorem \ref{th1.2} by letting $R\to\infty$. Moreover, with the method inspired by \cite{HanWinterWang25} and \cite{SungAnnaWang14}, we get an explicit global estimate on the complete noncompact Riemannian manifold without boundary.
\begin{thm}\label{th1.4}
	Let $(M^n.g)$ be a complete Riemannian manifold without boundary. Assume the Bakry-$\acute{E}$mery curvature satisfies $Ric_V\geq -(n-1)K$ on $M$, and $\lvert V\rvert\leq\theta$. Suppose $u$ is a positive solution to (\ref{1.1}) on $M$, then for any $x\in M$ we have
	\begin{equation*}
		\frac{\lvert\nabla u\rvert}{u}\leq\sqrt{(2n-1)\bigl((n-1)K+\frac{\theta^2}{n}\bigr)}.
	\end{equation*}
\end{thm}

\qquad  The remainder of this paper is organized as follows. Section 2 provides a brief introduction to several concepts related to $(p,V)$-harmonic functions and the Bakry-Emery curvature, etc. This is followed by a derivation of the Sobolev embedding theorem, which is crucial for the Moser iteration technique. Section 3 derives the local gradient estimate via the Moser iteration method. We give the proof of Theorem \ref{th1.4} in section 4.
\section{Preliminaries}
\subsection{$(p,V)$-harmonic function and Bakry-$\bf\acute{E}$mery curvature}
\qquad Let $(M^n,g)$ be a complete Riemannian manifold. For any $p>1$, a $p$-harmonic function is the solution to
\begin{equation*}
	\begin{aligned}
		0=\Delta_p u=&div(\lvert\nabla u\rvert^{p-2}\nabla u)\\
		=&\lvert\nabla u\rvert^{p-2}\Delta u+(\frac{p}{2}-1)\lvert\nabla u\rvert^{p-4}\langle \nabla u,\nabla\lvert\nabla u\rvert^2\rangle,
	\end{aligned}
\end{equation*}
where $\langle\ ,\ \rangle$ denotes the inner product of two vectors. The $p$-harmonic function is a natural generalization of the harmonic function, and it is the critical point of the $p$-energy functional 
\[E_p(u)=\frac{1}{p}\int\lvert\nabla u\rvert^p dvol_g,\]
where $dvol_g$ is the volume element introduced by $g$.

\qquad Let $V$ be a smooth vector field on $M$. The $(p,V)$-harmonic function is the solution to
\begin{equation}\label{2.1.1}
	\begin{aligned}
		0=\Delta_{p,V}u=&\Delta_pu+\lvert \nabla u\rvert^{p-2}\langle V,\nabla u\rangle\\
		=&\lvert\nabla u\rvert^{p-2}\Delta u+(\frac{p}{2}-1)\lvert\nabla u\rvert^{p-4}\langle \nabla u,\nabla\lvert\nabla u\rvert^2\rangle+\lvert \nabla u\rvert^{p-2}\langle V,\nabla u\rangle.
	\end{aligned}
\end{equation}
When $V$ is a gradient vector field, i.e., there exists a smooth function $f$ such that $V=-\nabla f$, (\ref{2.1.1}) becomes
\begin{equation}\label{2.1.2}
	\begin{aligned}
			0=\Delta_{p,f} u=&e^fdiv(e^{-f}\lvert\nabla u\rvert^{p-2}\nabla u)\\
		=&\lvert\nabla u\rvert^{p-2}\Delta u-\lvert\nabla u\rvert^{p-2}\langle\nabla f,\nabla u\rangle+(\frac{p}{2}-1)\lvert\nabla u\rvert^{p-4}\langle \nabla u,\nabla\lvert\nabla u\rvert^2\rangle.
	\end{aligned}
\end{equation}
A solution of (\ref{2.1.2}) is called a $(p,f)$-harmonic function. Zhao and Yang \cite{ZhaoYang18} got a gradient estimate for a $(p,f)$-Lichnerowicz equation through the Moser iteration method in 2018. For recent studies involving $(p,f)$-Laplace operator or $(p,V)$-Laplace operator, see \cite{DaiDungTuyenZhao22,WangLF18,WangYangChen13} and the references therein.

\qquad The Bakry-$\rm\acute{E}$mery curvature, named after Bakry and $\rm\acute{E}$mery \cite{BakryEmery85}, plays a vital role in the analysis of the $V$-Laplace operator. Following the notation established in \cite{ChenJostQiu12}, we define the Bakry-$\rm\acute{E}$mery curvature by
\[Ric_V=Ric-\frac{1}{2}L_Vg,\]
where $Ric$ is the Ricci curvature and $L_Vg$ is the Lie derivative with respect to the vector field $V$. 

\subsection{Laplacian and volume comparison theorems}
\qquad In this subsection, we will derive a Laplacian comparison theorem and a volume comparison theorem under the Bakry-$\rm\acute{E}$mery curvature condition. Our method is inspired by \cite{SongWuZhu25,WeiWylie09}. An alternative approach using Jacobi fields can be found in \cite{ChenJostQiu12}.
\begin{thm}\label{th2.1}
	(Laplacian comparison) Let $(M^n,g)$ be a complete Riemannian manifold, and $o\in M$ be a fixed point. Let $r(x)=d (x,o)$ be the distance function between $x$ and $o$. Assume the Bakry-$\acute{E}$mery curvature satisfies $Ric_V\geq-(n-1)K$ on $B_o(2R)$, $K\geq0$. Suppose $\lvert V\rvert\leq \theta$ in $B_o(2R)$. Then in $B_o(2R)$ we have
	\begin{equation}\label{2.2.1}
		\Delta r\leq (n-1)\sqrt{K}+\frac{n-1}{r}+2\theta.
	\end{equation}
\end{thm}
\begin{proof}
	By applying the Bochner formula to the distance function $r(x)$, we get
	\begin{equation*}
		\begin{aligned}
				0=\frac{1}{2}\Delta\lvert\nabla r\rvert^2=&\lvert Hessr\rvert^2+\langle\nabla r,\nabla\Delta r\rangle+Ric(\nabla r,\nabla r)\\
				=&\lvert Hessr\rvert^2+(\Delta r)'+Ric(\partial r,\partial r),
		\end{aligned}
	\end{equation*}
	where $Hessr$ is the Hessian of $r$ and the derivative is taken with respect to $r$. It is known that $\Delta r$ is the mean curvature $m$ of $\partial B_o(r)$, therefore we have
	\begin{equation*}
		m'\leq-\frac{m^2}{n-1}-Ric(\partial r,\partial r).
	\end{equation*}
	Let $M_{-K}$ be the $n$-dimensional model space with constant sectional curvature $-K$, with its mean curvature denoted by $m_{-K}$. From the property of the model space we get
	\begin{equation*}
		m_{-K}'=-\frac{m_{-K}^2}{n-1}+(n-1)K.
	\end{equation*} 
	Then we have
	\begin{equation*}
		m'-m_{-K}'\leq-\frac{m^2-m_{-K}^2}{n-1}-\frac{1}{2}L_Vg.
	\end{equation*}
	Denoting
	\begin{equation*}
		sn_{-K}(r)=
		\begin{cases}
			r,\ if\ K=0,\\
			\frac{sinh(\sqrt{K}r)}{\sqrt{K}},\ if\ K>0.
		\end{cases}
	\end{equation*}
	Then a direct computation yields
	\begin{equation*}
		\frac{sn_{-K}'}{sn_{-K}}=\frac{m_{-K}}{n-1},
	\end{equation*}
	and
	\begin{equation}\label{2.2.2}
		\begin{aligned}
			\bigl(sn_{-K}^2(m-m_{-K})\bigr)'=&2sn_{-K}'sn_{-K}(m-m_{-K})+sn_{-K}^2(m-m_{-K})'\\
			\leq&sn_{-K}^2\bigl(-\frac{(m-m_{-K})^2}{n-1}-\frac{1}{2}L_Vg\bigr)\\
			\leq&-\frac{1}{2}sn_{-K}^2L_Vg.
		\end{aligned}
	\end{equation}
	Let $\gamma$ be a unit speed geodesic from $o$ to $x$ with $\gamma(0)=o$. By integrating (\ref{2.2.2}) from $0$ to $r$ we obtain 
	\begin{equation}\label{2.2.3}
		sn_{-K}^2(r)\bigl(m(r)-m_{-K}(r)\bigr)\leq-\frac{1}{2}\int_{0}^{r} sn_{-K}^2(t)L_Vg(\nabla\gamma(t),\nabla \gamma(t))dt.
	\end{equation}
	Choosing a parallel frame $\{e_i\}$ along $\gamma$ such that $e_1=\nabla \gamma$, then we have (see \cite{ChenQiu16})
	\begin{equation}\label{2.2.4}
		L_Vg(\nabla\gamma,\nabla\gamma)=2V_1^1.
	\end{equation}
	Substituting (\ref{2.2.4}) into (\ref{2.2.3}), with the integration by parts formula we get 
	\begin{equation*}
		\begin{aligned}
			sn_{-K}^2(r)\bigl(m(r)-m_{-K}(r)\bigr)\leq&-sn_{-K}^2(r)V^1(\gamma(r))+sn_{-K}^2(0)V^1(\gamma(0))\\
			&+\int_{0}^{r}(sn_{-K}^2(t))'V^1(\gamma(t))dt.
		\end{aligned}
	\end{equation*}
	Then from the definition of $sn_{-K}$ and the assumption of $V$, by integration by parts we have
	\begin{equation*}
		m(r)\leq m_{-K}(r)+2\theta.
	\end{equation*}
	From the definition of $sn_{-K}$ we know the inequality (\ref{2.2.1}) holds.
\end{proof}
\qquad Under the spherical coordinate centered at $o$, it is known that
\begin{equation*}
	m(r)=\Delta r=\frac{J'(r,\psi,o)}{J(r,\psi,o)},
\end{equation*}
where $J$ is the volume element. By integration we have the volume comparison theorem. We omit the detailed proof since it is standard.
\begin{thm}\label{th2.2}
	 Let $(M^n,g)$ be a complete Riemannian manifold. Assume the Bakry-$\acute{E}$mery curvature satisfies $Ric_V\geq-(n-1)K$ on $B_o(2R)$, $K\geq0$. Suppose $\lvert V\rvert\leq \theta$ in $B_o(2R)$. Then for any $x\in B_o(R)$ and $0<r_1\leq r_2<d(x,\partial B_o(2R))$ we have
	
	$\ \ (i)$ 
	\[\frac{J(r_2,\psi,x)}{J(r_1,\psi,x)}\leq (\frac{r_2}{r_1})^{n-1}exp\bigl(((n-1)\sqrt{K}+2\theta)(r_2-r_1)\bigr).\]
	$\ \ (ii)$
	\[\frac{Vol(B_x(r_2))}{Vol(B_x(r_1))}\leq (\frac{r_2}{r_1})^n exp\bigl(((n-1)\sqrt{K}+2\theta)(r_2-r_1)\bigr),\]
	where $Vol$ is the volume.
\end{thm}
\qquad From the volume comparison theorem we know that $M$ satisfies the following volume doubling property.
\begin{cor}\label{cor2.3}
	Let $(M^n,g)$ be a complete Riemannian manifold. Assume the Bakry-$\acute{E}$mery curvature satisfies $Ric_V\geq-(n-1)K$ on $B_o(2R)$, $K\geq0$. Suppose $\lvert V\rvert\leq \theta$ in $B_o(2R)$. Assume $x\in B_o(R)$ and $0<r<\frac{1}{2}d(x,\partial B_o(2R))$. Then there exists a constant $C_1=C_1(n)$ such that
	\[\frac{Vol(B_x(2r))}{Vol(B_x(r))}\leq C_1exp\bigl((\sqrt{K}+\theta)r\bigr).\]
\end{cor}
\subsection{Sobolev embedding theorem}
\qquad In this subsection we derive a Sobolev embedding theorem following the method in \cite{SongWuZhu23,SongWuZhu25,WuWu15}.

\qquad By the volume comparison theorem, we can get a local Neumann Poincar$\rm\acute{e}$ inequality.
\begin{lemma}\label{le2.4}
	Let $(M^n,g)$ be a complete Riemannian manifold. Assume the Bakry-$\acute{E}$mery curvature satisfies $Ric_V\geq -(n-1)K$ on $B_o(2R)$, $K\geq0$, and $\lvert V\rvert\leq \theta$ in $B_o(2R)$. Then for any $x\in B_o(R)$ and $0<r< d(x,\partial B_o(2R))$, there exist $C_2,C_3$ depending on $n$ such that
	\begin{equation*}
		\int_{B_x(r)} \lvert\varphi-\varphi_{B_x(r)}\rvert^2 dvol_g\leq C_2r^2exp\bigl(C_3(\sqrt{K}+\theta)r\bigr)\int_{B_x(r)} \lvert\nabla \varphi\rvert^2 dvol_g,
	\end{equation*}
	for any $\varphi\in C^{\infty}_0$, where $\varphi_{B_x(r)}$ is the mean value of $\varphi$ in $B_x(r)$, i.e.,
	\[\varphi_{B_x(r)}=Vol^{-1}(B_x(r))\int_{B_x(r)}\varphi dvol_g.\]
\end{lemma}
\begin{proof}
	Lemma \ref{le2.4} can be obtained by following Buser's proof \cite{Buser82} or Saloff-Coste’s alternate proof (\cite{SaloffCoste02}, Theorem 5.6.5), see also \cite{MunteanuWang11}.
\end{proof}
\qquad Then we give a local Neumann Sobolev inequality.
\begin{lemma}\label{le2.5}
	Let $(M^n,g)$ be a complete Riemannian manifold. Assume the Bakry-$\acute{E}$mery curvature satisfies $Ric_V\geq -(n-1)K$ on $B_o(2R)$, $K\geq0$, and $\lvert V\rvert\leq \theta$ in $B_o(2R)$. Then for any $0<r\leq R$, there exist constants $C_4,C_5$ depending on $n$ such that
	\begin{equation*}
		\biggl(\int_{B_o(r)}\lvert \varphi-\varphi_{B_o(r)}\rvert^{\frac{2\mu}{\mu-2}}dvol_g\biggr)^{\frac{\mu-2}{\mu}}\leq\frac{C_4exp\bigl(C_5(\sqrt{K}+\theta)r\bigr)}{Vol^{\frac{2}{\mu}}(B_o(r))}r^2\int_{B_o(r)}\lvert\nabla\varphi\rvert^2dvol_g,
	\end{equation*}
	where $\varphi\in C^{\infty}_0$, and $\mu=4n-2$.
\end{lemma}
\begin{proof}
	The proof of Lemma \ref{le2.5} is similar to that of Lemma 3.2 in \cite{MunteanuWang11}.
\end{proof}
\qquad Finally, we can obtain the following Sobolev embedding theorem with the Minkowski inequality.
\begin{lemma}\label{le2.6}
	Let $(M^n,g)$ be a complete Riemannian manifold. Assume the Bakry-$\acute{E}$mery curvature satisfies $Ric_V\geq -(n-1)K$ on $B_o(2R)$, $K\geq0$, and $\lvert V\rvert\leq \theta$ in $B_o(2R)$. Then for any $\varphi\in C^{\infty}_0(B_o(r))$, $0<r\leq R$, and $\mu=4n-2$, there exist constants $C_6,C_7$ depending on $n$ such that
	\begin{equation*}
		\biggl(\int_{B_o(r)}\lvert \varphi\rvert^{\frac{2\mu}{\mu-2}} dvol_g\biggr)^{\frac{\mu-2}{\mu}}\leq \frac{C_6exp\bigl(C_7(\sqrt{K}+\theta)r\bigr)}{Vol^{\frac{2}{\mu}}(B_o(r))}r^2\int_{B_o(r)}\Bigl(\lvert\nabla\varphi\rvert^2+r^{-2}\varphi^2\Bigr)dvol_g.
	\end{equation*}
\end{lemma}
\section{Local estimate}
\qquad Let $u$ be a positive $(p,V)$-harmonic function in the weak sense, i.e., a positive weak solution to (\ref{1.1}). Denoting $w=-(p-1)log u$, then $w$ satisfies
\begin{equation}\label{3.1}
	\Delta_{p,V} w=\lvert\nabla w\rvert^p.
\end{equation}
\qquad We define the operator $\mathcal{L}$ by
\begin{equation*}
	\begin{aligned}
		\mathcal{L}\varphi=& \lvert\nabla w\rvert^{p-2}\Delta \varphi+\lvert\nabla w\rvert^{p-2}\langle V,\nabla \varphi\rangle+(\frac{p}{2}-1)\lvert\nabla w\rvert^{p-4}\langle \nabla \varphi,\nabla\lvert\nabla w\rvert^2\rangle\\
		&+(p-2)\lvert\nabla w\rvert^{p-4}\langle\nabla w,\nabla \varphi\rangle\Delta w+(p-2)\lvert\nabla w\rvert^{p-4}\langle\nabla w,\nabla \varphi\rangle\langle V,\nabla w\rangle\\
		&+(p-2)(\frac{p}{2}-2)\lvert\nabla w\rvert^{p-6}\langle\nabla w,\nabla \varphi\rangle\langle\nabla w,\nabla\lvert\nabla w\rvert^2\rangle\\
		&+(p-2)\lvert\nabla w\rvert^{p-4}\langle\nabla w,\nabla\langle\nabla w,\nabla \varphi\rangle\rangle-p\lvert\nabla w\rvert^{p-2}\langle\nabla w,\nabla \varphi\rangle.
	\end{aligned}
\end{equation*}
Note that when $V$ is a zero vector, the operator $\mathcal{L}$ is the same as that in \cite{WangZhang11}.

\qquad Denoting $h=\lvert\nabla w\rvert^2$, then we have the following lemma.
\begin{lemma}\label{le3.1}
	Let $(M^n,g)$ be a complete Riemannian manifold, $w$ be a positive solution to (\ref{3.1}), and $h=\lvert\nabla w\rvert^2$. Then we have
	\begin{equation}\label{3.2}
		\mathcal{L}h=2h^{\frac{p}{2}-1}(\lvert Hessw\rvert^2+Ric_V(\nabla w,\nabla w))+(\frac{p}{2}-1)h^{\frac{p}{2}-2}\lvert\nabla h\rvert^2.
	\end{equation}
\end{lemma}
\begin{proof}
	For completeness we provide the detailed proof, which follows by direct computation. From \cite{ChenQiu16} we have the following Bochner formula for the $V$-Laplace operator.
	\begin{equation}\label{3.3}
		\frac{1}{2}\Delta_V\lvert \nabla \varphi\rvert^2=\lvert Hess\varphi\rvert^2+\langle \nabla \varphi,\nabla \Delta_V\varphi\rangle+Ric_V(\nabla \varphi,\nabla \varphi).
	\end{equation}
	
	From the definition of $\mathcal{L}$ and $h$ we have
	\begin{equation}\label{3.4}
		\begin{aligned}
			\mathcal{L}h=& (\frac{p}{2}-1)h^{\frac{p}{2}-2}\lvert\nabla h\rvert^2+h^{\frac{p}{2}-1}\Delta h+h^{\frac{p}{2}-1}\langle V,\nabla h\rangle\\
			&+(p-2)h^{\frac{p}{2}-2}\langle\nabla w,\nabla h\rangle\Delta w+(p-2)h^{\frac{p}{2}-2}\langle\nabla w,\nabla h\rangle\langle V,\nabla w\rangle\\
			&+(p-2)(\frac{p}{2}-2)h^{\frac{p}{2}-3}\langle\nabla w,\nabla h\rangle^2+(p-2)h^{\frac{p}{2}-2}\langle\nabla w,\nabla\langle\nabla w,\nabla h\rangle\rangle\\
			&-ph^{\frac{p}{2}-1}\langle\nabla w,\nabla h\rangle\\
			=&(\frac{p}{2}-1)h^{\frac{p}{2}-2}\lvert\nabla h\rvert^2+h^{\frac{p}{2}-1}\Delta_Vh+(p-2)h^{\frac{p}{2}-2}\langle\nabla w,\nabla h\rangle\Delta_Vw\\
			&+(p-2)(\frac{p}{2}-2)h^{\frac{p}{2}-3}\langle\nabla w,\nabla h\rangle^2+(p-2)h^{\frac{p}{2}-2}\langle\nabla w,\nabla\langle\nabla w,\nabla h\rangle\rangle\\
			&-ph^{\frac{p}{2}-1}\langle\nabla w,\nabla h\rangle.
		\end{aligned}
	\end{equation}
	 Taking the gradient on both sides of (\ref{3.1}) and computing its product with $\nabla w$, then we obtain
	\begin{equation}\label{3.5}
		\begin{aligned}
			\frac{p}{2}h^{\frac{p}{2}-1}\langle\nabla w,\nabla h\rangle=&(\frac{p}{2}-1)h^{\frac{p}{2}-2}\langle\nabla w,\nabla h\rangle \Delta_Vw+h^{\frac{p}{2}-1}\langle\nabla w,\nabla\Delta_Vw\rangle\\
			&+(\frac{p}{2}-1)(\frac{p}{2}-2)h^{\frac{p}{2}-3}\langle\nabla w,\nabla h\rangle^2\\
			&+(\frac{p}{2}-1)h^{\frac{p}{2}-2}\langle\nabla w,\nabla\langle\nabla w,\nabla h\rangle\rangle.
		\end{aligned}
	\end{equation}
	The equation (\ref{3.2}) is proven by combining (\ref{3.3})--(\ref{3.5}).
\end{proof}
\qquad For any point in $B_o(2R)$, choosing an orthonormal frame $\{e_i\}$ such that $\lvert\nabla w\rvert e_1=\nabla w$. Then (\ref{3.1}) becomes
\begin{equation*}
	h=(p-1)w_{11}+\sum_{i=2}^{n}w_{ii}+\langle V,\nabla w\rangle,
\end{equation*}
and we have
\[\langle\nabla w,\nabla h\rangle=2hw_{11},\]
\[4\sum_{i=1}^nw_{1i}^2=\frac{\lvert\nabla h\rvert^2}{h}.\]
Therefore we can deduce 
\begin{equation}\label{3.5.2}
	\begin{aligned}
		\lvert Hessw\rvert^2\geq&w_{11}^2+2\sum_{i=2}^{n}w_{1i}^2+\sum_{i=2}^{n}w_{ii}^2\\
		\geq&w_{11}^2+2\sum_{i=2}^{n}w_{1i}^2+\frac{1}{n-1}(\sum_{i=2}^nw_{ii})^2\\
		=&w_{11}^2+2\sum_{i=2}^{n}w_{1i}^2+\frac{1}{n-1}\bigl(h-(p-1)w_{11}-\langle V,\nabla w\rangle\bigr)^2\\
		\geq&w_{11}^2+2\sum_{i=2}^nw_{1i}^2+\frac{(h-(p-1)w_{11})^2}{2n-1}-\frac{\langle V,\nabla w\rangle^2}{n}\\
		\geq&\frac{h^2}{2n-1}-\frac{2(p-1)hw_{11}}{2n-1}+c_1\sum_{i=1}^nw_{1i}^2-\frac{\langle V,\nabla w\rangle^2}{n},
	\end{aligned}
\end{equation}
where $c_1=1+min\{1,\frac{(p-1)^2}{2n-1}\}$, and we used the basic inequality $(\alpha+\beta)^2\geq\frac{\alpha^2}{1+\epsilon}-\frac{\beta^2}{\epsilon}$ for any $\epsilon>0$. From the curvature assumption we obtain
\begin{equation}\label{3.6}
	\begin{aligned}
		\mathcal{L}h\geq&-2(n-1)Kh^{\frac{p}{2}}+(\frac{p+c_1}{2}-1)\lvert\nabla h\rvert^2h^{\frac{p}{2}-2}\\
		&+\frac{2}{2n-1}h^{\frac{p}{2}+1}-\frac{2(p-1)}{2n-1}h^{\frac{p}{2}-1}\langle\nabla w,\nabla h\rangle-\frac{2\theta^2}{n}h^{\frac{p}{2}}\\
		\geq&\frac{2}{2n-1}h^{\frac{p}{2}+1}-2\bigl((n-1)K+\frac{\theta^2}{n}\bigr)h^\frac{p}{2}-\frac{2(p-1)}{2n-1}h^{\frac{p}{2}-1}\langle\nabla w,\nabla h\rangle.
	\end{aligned}
\end{equation}
\qquad The inequality (\ref{3.6}) holds wherever $w$ is strictly positive. Denoting a subset of $B_o(2R)$ by $E$ which satisfies
\[E=\{x\in B_o(2R): w(x)=0\}.\]
Therefore for any nonnegative cutoff function $\phi$ which has compact support in $B_o(2R)\backslash E$, we have
\begin{equation*}
	\begin{aligned}
		&\int_{B_o(2R)}\Bigl(h^{\frac{p}{2}-1}\langle \nabla h,\nabla \phi\rangle+(p-2)h^{\frac{p}{2}-2}\langle\nabla w,\nabla h\rangle\langle\nabla w,\nabla\phi\rangle\\
		&-h^{\frac{p}{2}-1}\langle V,\nabla h\rangle\phi-(p-2)h^{\frac{p}{2}-2}\langle\nabla w,\nabla h\rangle\langle V,\nabla w\rangle\phi\\
		&+ph^{\frac{p}{2}-1}\langle\nabla w,\nabla h\rangle\phi+\frac{2}{2n-1}h^{\frac{p}{2}+1}\phi\Bigr)\\
		&\leq\int_{B_o(2R)}\Bigl(\bigl(2(n-1)K+\frac{2\theta^2}{n}\bigr)h^{\frac{p}{2}}\phi+\frac{2(p-1)}{2n-1}h^{\frac{p}{2}-1}\langle\nabla w,\nabla h\rangle\phi\Bigr).
	\end{aligned}
\end{equation*}
\qquad Let $\phi=h^b_{\epsilon}\eta^2$, where $b>1$ is a constant to be determined later, $\epsilon>0$,  $h_{\epsilon}=max\{h-\epsilon,0\}$, and $\eta\in C^{\infty}_0(B_o(2R))$ is nonnegative. Then by direct computation we have
\begin{equation*}
	\nabla\phi=bh_{\epsilon}^{b-1}\eta^2\nabla h+2h_{\epsilon}^b\eta\nabla\eta,
\end{equation*}
and
\begin{equation}\label{3.7}
	\begin{aligned}
		&\int_{B_o(2R)}\Bigl(bh^{\frac{p}{2}-1}h_{\epsilon}^{b-1}\eta^2\lvert\nabla h\rvert^2+b(p-2)h^{\frac{p}{2}-2}h_{\epsilon}^{b-1}\eta^2\langle\nabla w,\nabla h\rangle^2\\
		&+2h^{\frac{p}{2}-1}h_{\epsilon}^b\eta\langle\nabla h,\nabla \eta\rangle+2(p-2)h^{\frac{p}{2}-2}h_{\epsilon}^b\eta\langle\nabla w,\nabla h\rangle\langle\nabla w,\nabla \eta\rangle\\
		&-h^{\frac{p}{2}-1}h_{\epsilon}^b\eta^2\langle V,\nabla h\rangle-(p-2)h^{\frac{p}{2}-2}h_\epsilon^b\eta^2\langle\nabla w,\nabla h\rangle\langle V,\nabla w\rangle\\
		&+ph^{\frac{p}{2}-1}h_{\epsilon}^b\eta^2\langle \nabla w,\nabla h\rangle+\frac{2}{2n-1}h^{\frac{p}{2}+1}h_{\epsilon}^b\eta^2\Bigr)\\
		\leq&\int_{B_o(2R)}\Bigl(\bigl(2(n-1)K+\frac{2\theta^2}{n}\bigr)h^{\frac{p}{2}}h_{\epsilon}^b\eta^2+\frac{2(p-1)}{2n-1}h^{\frac{p}{2}-1}h_{\epsilon}^b\eta^2\langle\nabla w,\nabla h\rangle\Bigr).
	\end{aligned}
\end{equation}
For the left hand side of (\ref{3.7}), we get
\begin{equation*}
		bh^{\frac{p}{2}-1}h_{\epsilon}^{b-1}\eta^2\lvert\nabla h\rvert^2+b(p-2)h^{\frac{p}{2}-2}h_{\epsilon}^{b-1}\eta^2\langle\nabla w,\nabla h\rangle^2\geq bc_2h^{\frac{p}{2}-1}h_{\epsilon}^{b-1}\eta^2\lvert\nabla h\rvert^2,
\end{equation*}
and
\begin{equation*}
	\begin{aligned}
		-2(p+1)h^{\frac{p}{2}-1}h_{\epsilon}^b\eta\lvert\nabla h\rvert\lvert\nabla\eta\rvert\leq&2h^{\frac{p}{2}-1}h_{\epsilon}^b\eta\langle\nabla h,\nabla \eta\rangle\\
		&+2(p-2)h^{\frac{p}{2}-2}h_{\epsilon}^b\eta\langle\nabla w,\nabla h\rangle\langle\nabla w,\nabla \eta\rangle,
	\end{aligned}
\end{equation*}
where $c_2=\min\{1,p-1\}$. Therefore by letting $\epsilon\to0$ we obtain
\begin{equation}\label{3.8}
	\begin{aligned}
		&\int_{B_o(2R)}\Bigl(bc_2h^{\frac{p}{2}+b-2}\eta^2\lvert\nabla h\rvert^2-h^{\frac{p}{2}+b-1}\eta^2\langle V,\nabla h\rangle\\
		&-(p-2)h^{\frac{p}{2}+b-2}\eta^2\langle\nabla w,\nabla h\rangle\langle V,\nabla w\rangle+ph^{\frac{p}{2}+b-1}\eta^2\langle \nabla w,\nabla h\rangle\\
		&+\frac{2}{2n-1}h^{\frac{p}{2}+b+1}\eta^2\Bigr)\\
		\leq&\int_{B_o(2R)}\Bigl(\bigl(2(n-1)K+\frac{2\theta^2}{n}\bigr)h^{\frac{p}{2}+b}\eta^2+\frac{2(p-1)}{2n-1}h^{\frac{p}{2}+b-1}\eta^2\langle\nabla w,\nabla h\rangle\\
		&+2(p+1)h^{\frac{p}{2}+b-1}\eta\lvert\nabla h\rvert\lvert\nabla\eta\rvert\Bigr).
	\end{aligned}
\end{equation}
\qquad From now on we use $c_i$ to denote constants depending only on $p$ and $n$. Combining (\ref{3.8}) and the assumption that $\lvert V\rvert\leq \theta$, after rearranging the terms we get
\begin{equation}\label{3.9}
	\begin{aligned}
		&\int_{B_o(2R)}\Bigl((bc_2-c_3)h^{\frac{p}{2}+b-2}\eta^2\lvert\nabla h\rvert^2+\frac{2}{2n-1}h^{\frac{p}{2}+b+1}\eta^2\Bigr)\\
		\leq&\int_{B_o(2R)}\Bigl(\bigl(2(n-1)K+c_3\theta^2\bigr)h^{\frac{p}{2}+b}\eta^2+c_4h^{\frac{p}{2}+b-\frac{1}{2}}\eta^2\lvert\nabla h\rvert\\
		&+c_5h^{\frac{p}{2}+b-1}\eta\lvert \nabla h\rvert\lvert\nabla \eta\rvert\Bigr).
	\end{aligned}
\end{equation}
The Schwarz inequality yields
\begin{equation*}
	c_4h^{\frac{p}{2}+b-\frac{1}{2}}\eta^2\lvert\nabla h\rvert\leq \frac{bc_2}{6}h^{\frac{p}{2}+b-2}\eta^2\lvert\nabla h\rvert^2+\frac{c_6}{b}h^{\frac{p}{2}+b+1}\eta^2,
\end{equation*}
and
\begin{equation*}
	c_5h^{\frac{p}{2}+b-1}\eta\lvert \nabla h\rvert\lvert\nabla \eta\rvert\leq \frac{bc_2}{6}h^{\frac{p}{2}+b-2}\eta^2\lvert\nabla h\rvert^2+\frac{c_7}{b}h^{\frac{p}{2}+b}\lvert\nabla\eta\rvert^2.
\end{equation*}
Choosing $b$ large enough such that
\begin{equation}\label{3.10}
	c_3\leq\frac{bc_2}{6},\ \frac{c_6}{b}\leq\frac{1}{2n-1},
\end{equation}
then (\ref{3.9}) becomes
\begin{equation}\label{3.11}
	\begin{aligned}
		&\int_{B_o(2R)}\Bigl(\frac{bc_2}{2}h^{\frac{p}{2}+b-2}\eta^2\lvert\nabla h\rvert^2+\frac{1}{2n-1}h^{\frac{p}{2}+b+1}\eta^2\Bigr)\\
		\leq&\int_{B_o(2R)}\Bigl(\bigl(2(n-1)K+c_3\theta^2\bigr)h^{\frac{p}{2}+b}\eta^2+\frac{c_7}{b}h^{\frac{p}{2}+b}\lvert\nabla\eta\rvert^2\Bigr).
	\end{aligned}
\end{equation}
Since
\begin{equation*}
	\lvert\nabla(h^{\frac{p}{4}+\frac{b}{2}}\eta)\rvert^2\leq \frac{1}{2}(\frac{p}{2}+b)^2h^{\frac{p}{2}+b-2}\eta^2\lvert\nabla h\rvert^2+2h^{\frac{p}{2}+b}\lvert\nabla\eta\rvert^2,
\end{equation*}
we can obtain from (\ref{3.11}) that
\begin{equation}\label{3.12}
	\begin{aligned}
		&\int_{B_o(2R)}\Bigl(\lvert\nabla(h^{\frac{p}{4}+\frac{b}{2}}\eta)\rvert^2+c_8bh^{\frac{p}{2}+b+1}\eta^2\Bigr)\\
		&\leq \int_{B_o(2R)}\Bigl(c_9b(K+\theta^2)h^{\frac{p}{2}+b}\eta^2+c_{10}h^{\frac{p}{2}+b}\lvert\nabla\eta\rvert^2\Bigr).
	\end{aligned}
\end{equation}
\qquad Let $b_0=C_0(n,p)\bigl(1+(\sqrt{K}+\theta)R\bigr)$ which is sufficient large such that (\ref{3.10}) holds. Then from Lemma \ref{le2.6} we have
\begin{equation}\label{3.13}
	\begin{aligned}
		&\biggl(\int_{B_o(2R)}h^{(\frac{p}{2}+b)\frac{\mu}{\mu-2}}\eta^{\frac{2\mu}{\mu-2}}\biggr)^\frac{\mu-2}{\mu}+c_{11}bR^2\frac{exp(c_{12}b_0)}{Vol^{\frac{2}{\mu}}(B_o(2R))}\int_{B_o(2R)}h^{\frac{p}{2}+b+1}\eta^2\\
		\leq&c_{13}b_0^2b\frac{exp(c_{12}b_0)}{Vol^{\frac{2}{\mu}}(B_o(2R))}\int_{B_o(2R)}h^{\frac{p}{2}+b}\eta^2+c_{14}R^2\frac{exp(c_{12}b_0)}{Vol^{\frac{2}{\mu}}(B_o(2R))}\int_{B_o(2R)}h^{\frac{p}{2}+b}\lvert\nabla\eta\rvert^2,
	\end{aligned}
\end{equation}
where $\mu=4n-2$.

\qquad Denoting $b_1=(b_0+\frac{p}{2})\frac{\mu}{\mu-2}$. Now we estimate the $L^{b_1}$-norm of $h$.
\begin{lemma}\label{le3.2}
	Let $b_0$ be a sufficient large constant, $\mu=4n-2$, and $b_1=(b_0+\frac{p}{2})\frac{\mu}{\mu-2}$. Then for $R>0$, there exists a constant $C=C(n,p)$ such that
		\begin{equation*}
		\lVert h\rVert_{L^{b_1}(B_o(\frac{3}{2}R))}\leq CVol^{\frac{1}{b_1}}(B_o(2R))\frac{b_0^2}{R^2}.
	\end{equation*}
\end{lemma}
\begin{proof}
	By observation we have
	\[c_{13}b_0^2bh^{\frac{p}{2}+b}<\frac{1}{2}c_{11}bR^2h^{\frac{p}{2}+b+1},\]
	if
	\[h>\frac{2c_{13}b_0^2}{c_{11}R^2}.\]
	Thus by decomposing (\ref{3.13}) and letting $b=b_0$ we can obtain that
	\begin{equation}\label{3.14}
		\begin{aligned}
				&\biggl(\int_{B_o(2R)}h^{(\frac{p}{2}+b_0)\frac{\mu}{\mu-2}}\eta^{\frac{2\mu}{\mu-2}}\biggr)^\frac{\mu-2}{\mu}+\frac{c_{11}b_0R^2exp(c_{12}b_0)}{2Vol^{\frac{2}{\mu}}(B_o(2R))}\int_{B_o(2R)}h^{\frac{p}{2}+b_0+1}\eta^2\\
				&\leq c_{14}R^2\frac{exp(c_{12}b_0)}{Vol^{\frac{2}{\mu}}(B_o(2R))}\int_{B_o(2R)}h^{\frac{p}{2}+b_0}\lvert\nabla\eta\rvert^2\\
				&+c_{15}^{b_0}b_0^3exp(c_{12}b_0)(\frac{b_0}{R})^{p+2b_0}Vol^{1-\frac{2}{\mu}}(B_o(2R)).
		\end{aligned}
	\end{equation}
	Let $\eta_1\in C^{\infty}_0B_o(2R)$ such that
	\begin{equation*}
		0\leq\eta_1\leq 1,\ \eta_1=1\ in\ B_o(\frac{3R}{2}),\ \lvert\nabla\eta_1\rvert\leq \frac{C}{R}.
	\end{equation*}
	Then for $\eta=\eta_1^{\frac{p}{2}+b_0+1}$, we have
	\begin{equation*}
		\begin{aligned}
			&c_{14}R^2\frac{exp(c_{12}b_0)}{Vol^{\frac{2}{\mu}}(B_o(2R))}\int_{B_o(2R)}h^{\frac{p}{2}+b_0}\lvert\nabla\eta\rvert^2\\
			\leq& c_{16}b_0^2\frac{exp(c_{12}b_0)}{Vol^{\frac{2}{\mu}}(B_o(2R))}\int_{B_o(2R)}h^{\frac{p}{2}+b_0}\eta^{\frac{p+2b_0}{\frac{p}{2}+b_0+1}}\\
			\leq&c_{16}b_0^2\frac{exp(c_{12}b_0)}{Vol^{\frac{2}{\mu}}(B_o(2R))}\biggl(\int_{B_o(2R)}h^{\frac{p}{2}+b_0+1}\eta^2\biggr)^{\frac{\frac{p}{2}+b_0}{\frac{p}{2}+b_0+1}}Vol^{\frac{1}{\frac{p}{2}+b_0+1}}(B_o(2R))\\
			\leq& \frac{c_{11}b_0R^2exp(c_{12}b_0)}{2Vol^{\frac{2}{\mu}}(B_o(2R))}\int_{B_o(2R)}h^{\frac{p}{2}+b_0+1}\eta^2+\frac{c_{17}^{b_0}b_0^{\frac{p}{2}+b_0+2}exp(c_{12}b_0)}{R^{p+2b_0}}Vol^{1-\frac{2}{\mu}}(B_o(2R)).
		\end{aligned}
	\end{equation*}
	Substituting it into (\ref{3.14}), then we can obtain
	\begin{equation}\label{3.15}
		(\int_{B_o(\frac{3}{2}R)}h^{(\frac{p}{2}+b_0)\frac{\mu}{\mu-2}}\eta^{\frac{2\mu}{\mu-2}})^\frac{\mu-2}{\mu}\leq C^{b_0}b_0^3exp(c_{12}b_0)(\frac{b_0}{R})^{p+2b_0}Vol^{1-\frac{2}{\mu}}(B_o(2R)).
	\end{equation} 
	Since $b_1=(b_0+\frac{p}{2})\frac{\mu}{\mu-2}$, we can get from (\ref{3.15}) that
	\begin{equation*}
		\lVert h\rVert_{L^{b_1}(B_o(\frac{3}{2}R))}\leq CVol^{\frac{1}{b_1}}(B_o(2R))\frac{b_0^2}{R^2}.
	\end{equation*}
\end{proof}
\qquad Now we can give the proof of Theorem \ref{th1.2}. 
\begin{proof}
	By ignoring the second term on the left hand side of (\ref{3.13}) we get
	\begin{equation}\label{3.16}
		\begin{aligned}
			\biggl(\int_{B_o(2R)}h^{(\frac{p}{2}+b)\frac{\mu}{\mu-2}}\eta^{\frac{2\mu}{\mu-2}}\biggr)^\frac{\mu-2}{\mu}\leq& c_{13}b_0^2b\frac{exp(c_{12}b_0)}{Vol^{\frac{2}{\mu}}(B_o(2R))}\int_{B_o(2R)}h^{\frac{p}{2}+b}\eta^2\\
			&+c_{14}R^2\frac{exp(c_{12}b_0)}{Vol^{\frac{2}{\mu}}(B_o(2R))}\int_{B_o(2R)}h^{\frac{p}{2}+b}\lvert\nabla\eta\rvert^2.
		\end{aligned}
	\end{equation}
	Denoting
	\[b_{i+1}=\frac{\mu}{\mu-2}b_i,\ R_i=R+\frac{2R}{4^i},\]
	and letting $\{\eta_i\}$ be cutoff functions such that
	\begin{equation*}
		0\leq\eta_i\leq1,\ \eta_i\in C^{\infty}_0(B_o(R_i)),\ \eta_i=1\ in\ B_o(R_{i+1}),\ \lvert\nabla\eta_i\rvert\leq\frac{C\cdot4^i}{R}.
	\end{equation*}
	Then by letting $\frac{p}{2}+b=b_i$ and $\eta=\eta_i$ in (\ref{3.16}) we obtain
	\begin{equation*}
		\begin{aligned}
			\biggl(\int_{B_o(R_{i+1})}h^{b_{i+1}}\biggr)^{\frac{1}{b_{i+1}}}\leq&\Bigl(\frac{c_{18}exp(c_{12}b_0)}{Vol^{\frac{2}{\mu}}(B_o(2R))}\Bigr)^{\frac{1}{b_i}}\biggl(\int_{B_o(R_{i})}\Bigl(b_0^2b_ih^{b_i}+R^2h^{b_i}\lvert\nabla\eta_i\rvert^2\Bigr)\biggr)^{\frac{1}{b_i}}
		\end{aligned}
	\end{equation*}
	The assumption of $\eta_i$ gives
	\begin{equation*}
		\lVert h\rVert_{L^{b_{i+1}}(B_o(R_{i+1}))}\leq \Bigl(\frac{c_{19}exp(c_{12}b_0)}{Vol^{\frac{2}{\mu}}(B_o(2R))}\Bigr)^{\frac{1}{b_i}}(b_0^2b_i+16^i)^{\frac{1}{b_i}}\lVert h\rVert _{L^{b_i}(B_o(R_i))}.
	\end{equation*}
	From direct computation we have 
	\begin{equation*}
		\sum_{i=1}^{\infty}\frac{1}{b_i}= \frac{\mu}{2b_1},\ \sum_{i=1}^{\infty}\frac{i}{b_i}=\frac{\mu^2}{4b_1}.
	\end{equation*}
	Therefore the iterative process yields
	\begin{equation}\label{3.17}
		\begin{aligned}
			\lVert h\rVert_{L^{\infty}(B_o(R))}\leq&\Bigl(\frac{c_{19}exp(c_{12}b_0)}{Vol^{\frac{2}{\mu}}(B_o(2R))}\Bigr)^{\sum_{i=1}^{\infty}\frac{1}{b_i}}\prod_{i=1}^{\infty}(b_0^2b_i+16^i)^{\frac{1}{b_i}}\lVert h\rVert_{L^{b_1}(B_o(\frac{3R}{2}))}\\
			\leq&\frac{c_{20}}{Vol^{\frac{1}{b_1}}(B_o(2R))}\lVert h\rVert_{L^{b_1}(B_o(\frac{3R}{2}))}.
		\end{aligned}
	\end{equation}
	By combining Lemma \ref{le3.2} and (\ref{3.17}) we get
	\[\lvert\nabla w\rvert^2=\lvert h\rvert\leq \frac{c_{21}b_0^2}{R^2},\]
	For any $x\in B_o(R)$. This completes the proof.
\end{proof}
\section{Global explicit estimate} 
\qquad In this section we give the proof of Theorem \ref{th1.4}. Without confusion, we use the same notations as in Section 3. 

\qquad We denote the linearized operator of $\Delta_{p,V}$ at $w$ by $\mathfrak{L}$, i.e.,
\begin{equation*}
	\begin{aligned}
		\mathfrak{L}\varphi=div (\lvert\nabla w\rvert^{p-2}A(\nabla\varphi))+\lvert\nabla w\rvert^{p-2}\langle V,A(\nabla \varphi)\rangle,
	\end{aligned}
\end{equation*}
where 
\[A=id+(p-2)\frac{dw\otimes dw}{\lvert\nabla w\rvert^2}.\]

\qquad By direct computation we have the following lemma.
\begin{lemma}\label{le4.1}
	Let $(M^n,g)$ be a complete noncompact Riemannian manifold without boundary. Assume $u$ is a positive solution to (\ref{1.1}), $w=-(p-1)logu$, and $h=\lvert\nabla w\rvert^2$. Then for any sufficient large constant $\alpha> 1$, we have
	\begin{equation}\label{4.1}
		\begin{aligned}
			\mathfrak{L}(h^\alpha)=&\alpha(\alpha+\frac{p}{2}-2)h^{\alpha+\frac{p}{2}-3}\lvert\nabla h\rvert^2+2\alpha h^{\alpha+\frac{p}{2}-2}(\lvert Hessw\rvert^2+Ric_V(\nabla w,\nabla w))\\
			&+2\alpha h^{\alpha-1}\langle\nabla\Delta_{p,V}w,\nabla w\rangle+\alpha(\alpha-1)(p-2)h^{\alpha+\frac{p}{2}-4}\langle\nabla w,\nabla h\rangle^2.
		\end{aligned}
	\end{equation}
\end{lemma}
\begin{proof}
	From the definition of $\mathfrak{L}$ we have
	\begin{equation}\label{4.2}
		\begin{aligned}
			\mathfrak{L}(h^\alpha)=&div(h^{\frac{p}{2}-1}\nabla h^\alpha)+h^{\frac{p}{2}-1}\langle V,\nabla h^\alpha\rangle+(p-2)div(h^{\frac{p}{2}-2}\langle \nabla w,\nabla h^\alpha\rangle\nabla w)\\
			&+(p-2)h^{\frac{p}{2}-2}\langle\nabla w,\nabla h^\alpha\rangle\langle V,\nabla w\rangle\\
			=&\alpha(\alpha+\frac{p}{2}-2)h^{\alpha+\frac{p}{2}-3}\lvert\nabla h\rvert^2+\alpha h^{\alpha+\frac{p}{2}-2}\Delta_V h\\
			&+\alpha(p-2)h^{\alpha+\frac{p}{2}-3}\langle \nabla h,\nabla w\rangle\Delta_Vw+\alpha(p-2)(\alpha+\frac{p}{2}-3)h^{\alpha+\frac{p}{2}-4}\langle\nabla w,\nabla h\rangle^2\\
			&+\alpha(p-2)h^{\alpha+\frac{p}{2}-3}\langle\nabla w,\nabla\langle\nabla w,\nabla h\rangle\rangle.
		\end{aligned}
	\end{equation}
	And we can infer from (\ref{2.1.1}) that
	\begin{equation}\label{4.3}
		\begin{aligned}
			\langle\nabla\Delta_{p,V}w,\nabla w\rangle=&(\frac{p}{2}-1)(\frac{p}{2}-2)h^{\frac{p}{2}-3}\langle\nabla w,\nabla h\rangle^2+(\frac{p}{2}-1)h^{\frac{p}{2}-2}\langle\nabla w,\nabla\langle \nabla w,\nabla h\rangle\rangle\\
			&+(\frac{p}{2}-1)h^{\frac{p}{2}-2}\langle\nabla w,\nabla h\rangle\Delta_Vw+h^{\frac{p}{2}-1}\langle\nabla\Delta_Vw,\nabla w\rangle.
		\end{aligned}
	\end{equation}
	Combining (\ref{4.2}), (\ref{4.3}), and the Bochner formula for the $V$-Laplace operator (\ref{3.3}), we obtain (\ref{4.1}).
\end{proof}
\qquad With a similar argument to that in \cite{HanWinterWang25}, we get the following estimate.
\begin{lemma}\label{le4.2}
	Under the same assumption as in Lemma \ref{le4.1}. Assume the Bakry-$\acute{E}$mery curvature satisfies $Ric_V\geq -(n-1)K$ on $M$, and $\lvert V\rvert\leq\theta$. Then for $\alpha\geq1$ sufficient large, on the point where $h$ is strictly positive, we have
	\begin{equation}\label{4.4}
		\begin{aligned}
			\mathfrak{L}(h^\alpha)\geq&\frac{2\alpha}{2n-1}h^{\alpha+\frac{p}{2}}-2\alpha\bigl((n-1)K+\frac{\theta^2}{n}\bigr)h^{\alpha+\frac{p}{2}-1}\\
			&-\alpha\bigl(\frac{2(p-1)}{2n-1}+p\bigr)h^{\alpha+\frac{p}{2}-\frac{3}{2}}\lvert\nabla h\rvert.
		\end{aligned}
	\end{equation}
\end{lemma}
\begin{proof}
	For any $x$ in $\{h(x)>0\}$, we denote an orthonormal frame $\{e_i\}_{i=1}^n$ at $x$, such that $\nabla w=\lvert\nabla w\rvert e_1$, then we get
	\[2hw_{11}=\langle\nabla w,\nabla h\rangle,\]
	\[\frac{\lvert\nabla h\rvert^2}{h}=4\sum_{i=1}^{n}w_{1i}^2,\]
	\[h=(p-1)w_{11}+\sum_{i=2}w_{ii}+\langle V,\nabla w\rangle.\]
	\qquad From Lemma \ref{le4.1} we have
	\begin{equation}\label{4.5}
		\begin{aligned}
			\frac{h^{2-\alpha-\frac{p}{2}}}{2\alpha}\mathfrak{L}(h^\alpha)=&\frac{1}{2}(\alpha+\frac{p}{2}-2)\frac{\lvert\nabla h\rvert^2}{h}+\lvert Hessw\rvert^2+Ric_V(\nabla w,\nabla w)\\
			&+\frac{1}{2}(\alpha-1)(p-2)\frac{\langle\nabla w,\nabla h\rangle^2}{h^2}+h^{1-\frac{p}{2}}\langle\nabla\Delta_{p,V}w,\nabla w\rangle\\
			=&\lvert Hessw\rvert^2+Ric_V(\nabla w,\nabla w)+2(\alpha+\frac{p}{2}-2)\sum_{i=1}^nw_{1i}^2\\
			&+2(\alpha-1)(p-2)w_{11}^2+h^{1-\frac{p}{2}}\langle\nabla\Delta_{p,V}w,\nabla w\rangle.
		\end{aligned}
	\end{equation}
	For the Hessian term, similar to (\ref{3.5.2}), we have
	\begin{equation}\label{4.6}
		\begin{aligned}
			\lvert Hessw\rvert^2
			\geq&\frac{h^2}{2n-1}-\frac{(p-1)\langle\nabla w,\nabla h\rangle}{2n-1}-\frac{\theta^2h}{n}\\
			&+\bigl(1+\frac{(p-1)^2}{2n-1}\bigr)w_{11}^2+2\sum_{i=2}^nw_{1i}^2.
		\end{aligned}
	\end{equation}
	The last term on the right hand side of (\ref{4.5}) can be represented as
	\begin{equation}\label{4.7}
		\begin{aligned}
			h^{1-\frac{p}{2}}\langle\nabla\Delta_{p,V}w,\nabla w\rangle=&h^{1-\frac{p}{2}}\langle\nabla h^{\frac{p}{2}},\nabla w\rangle\\
			=&\frac{p}{2}\langle\nabla w,\nabla h\rangle.
		\end{aligned}
	\end{equation}
	Combining (\ref{4.5})-(\ref{4.7}), with the curvature assumption we obtain
	\begin{equation*}
		\begin{aligned}
			\frac{h^{2-\alpha-\frac{p}{2}}}{2\alpha}\mathfrak{L}(h^\alpha)\geq&\frac{h^2}{2n-1}-\frac{p-1}{2n-1}\langle\nabla w,\nabla h\rangle-\frac{\theta^2h}{n}-(n-1)Kh\\
			&+2(\alpha+\frac{p}{2}-2)\sum_{i=2}^nw_{1i}^2+2(\alpha-1)(p-2)w_{11}^2\\
			&+\frac{p}{2}\langle\nabla w,\nabla h\rangle+\bigl(1+\frac{(p-1)^2}{2n-1}\bigr)w_{11}^2+2\sum_{i=2}^nw_{1i}^2\\
			&+2(\alpha+\frac{p}{2}-2)w_{11}^2\\
			\geq&\frac{h^2}{2n-1}-\bigl((n-1)K+\frac{\theta^2}{n}\bigr)h-(\frac{p-1}{2n-1}+\frac{p}{2})h^{\frac{1}{2}}\lvert\nabla h\rvert\\
			&+\bigl((2\alpha-1)(p-1)+\frac{(p-1)^2}{2n-1}\bigr)w_{11}^2+2(\alpha+\frac{p}{2}-1)\sum_{i=2}^nw_{1i}^2\\
			\geq&\frac{h^2}{2n-1}-\bigl((n-1)K+\frac{\theta^2}{n}\bigr)h-(\frac{p-1}{2n-1}+\frac{p}{2})h^{\frac{1}{2}}\lvert\nabla h\rvert.
		\end{aligned}
	\end{equation*}
	This is the wanted inequality (\ref{4.4}), and Lemma \ref{le4.2} is proven.
\end{proof}
\qquad We denote 
\[H=\Bigl(h-(2n-1)\bigl((n-1)K+\frac{\theta^2}{n}\bigr)-\delta\Bigr)^+,\ \delta>0,\]
and
\[D=\{x\in M:H(x)>0\}.\]
Then we have the following estimate for $H$.
\begin{lemma}\label{le4.3}
	Under the same assumption as in Lemma \ref{le4.2}, for any $\delta>0$, there exist two constants $d_1$ and $d_2$ such that
	\begin{equation}\label{4.8}
		\mathfrak{L}(H^\alpha)\geq 2\alpha H^{\alpha-1}(d_1H-d_2\lvert\nabla H\rvert).
	\end{equation}
\end{lemma}
\begin{proof}
	From the definition of $H$ we have $\nabla h=\nabla H$ and $h>H$ on the set $D$, which gives
	\begin{equation*}
		\begin{aligned}
			\mathfrak{L}(H^\alpha)=&\alpha div(h^{\frac{p}{2}-1}H^{\alpha-1}A(\nabla H))+\alpha h^{\frac{p}{2}-1}H^{\alpha-1}\langle V,A(\nabla H)\rangle\\
			=&\alpha div(h^{\frac{p}{2}-1}H^{\alpha-1}A(\nabla h))+\alpha h^{\frac{p}{2}-1}H^{\alpha-1}\langle V,A(\nabla h)\rangle\\
			\geq&\alpha H^{\alpha-1}\bigl((\alpha-1)h^{-1}\langle\nabla h,h^{\frac{p}{2}-1}A(\nabla h)\rangle+\mathfrak{L}(h)\bigr)\\
			=&\frac{H^{\alpha-1}}{h^{\alpha-1}}\mathfrak{L}(h^\alpha).
		\end{aligned}
	\end{equation*}
	This implies
	\begin{equation}\label{4.9}
		\begin{aligned}
			\mathfrak{L}(H^\alpha)\geq&2\alpha H^{\alpha-1}h^{\frac{p}{2}-1}\Bigl(\frac{h^2}{2n-1}-\bigl((n-1)K+\frac{\theta^2}{n}\bigr)h-(\frac{p-1}{2n-1}+\frac{p}{2})h^{\frac{1}{2}}\lvert\nabla H\rvert\Bigr)
		\end{aligned}
	\end{equation}
	on the set $D$. By choosing $\alpha>2$, the distribution on $\partial D$ is eliminated, and (\ref{4.9}) holds on $\bar{D}$. From Section 3 we know that $h$ is uniformly bounded, therefore Lemma \ref{le4.3} holds from the fact $h\geq(2n-1)\bigl((n-1)K+\frac{\theta^2}{n}\bigr)+\delta$ on $\bar{D}$.
\end{proof}
\qquad Now we are ready to prove Theorem \ref{th1.4}.
\begin{proof}
	We only need to prove $H=0$ on $M$ since the constant $\delta$ is arbitrary. By the method of contradiction, we may assume $H\neq 0$ on $B_o(1)$ for some point $o\in M$. From Lemma \ref{le4.3}, (\ref{4.8}) holds in the sense of distribution. And from now on we use $d_i$ to represent constants. By choosing the test function $H^\sigma\eta^2$, where $\eta\in C_0^{\infty}(D)$ and $\sigma$ is a constant to be determined later, we have
	\begin{equation}\label{4.10}
		\begin{aligned}
			&\int_M2\alpha H^{\alpha+\sigma-1}\eta^2(d_1H-d_2\lvert\nabla H\rvert)\\
			\leq&\int_M\Bigl(h^{\frac{p}{2}-1}\bigl(-\alpha\sigma H^{\alpha+\sigma-2}\eta^2\lvert\nabla H\rvert^2-\alpha\sigma(p-2)h^{-1}H^{\alpha+\sigma-2}\eta^2\langle\nabla H,\nabla w\rangle^2\\
			&-2\alpha H^{\alpha+\sigma-1}\eta\langle\nabla H,\nabla\eta\rangle-2\alpha(p-2)h^{-1}H^{\alpha+\sigma-1}\eta\langle\nabla H,\nabla w\rangle\langle\nabla w,\nabla \eta\rangle\\
			&+\alpha H^{\alpha+\sigma-1}\eta^2\langle V,\nabla H\rangle+\alpha(p-2)h^{-1}H^{\alpha+\sigma-1}\eta^2\langle V,\nabla w\rangle\langle \nabla H,\nabla w\rangle\bigr)\Bigr).
		\end{aligned}
	\end{equation}
	The function $h$ is uniformly bounded from both sides on $\bar{D}$, and on $M\backslash D$ we have $H=0$. Therefore with bounded assumption of $V$ the inequality (\ref{4.10}) becomes
	\begin{equation*}
		\begin{aligned}
			&\int_M2H^{\alpha+\sigma-1}\eta^2(d_1H-d_3\lvert\nabla H\rvert)+\int_M \sigma d_4 H^{\alpha+\sigma-2}\eta^2\lvert\nabla H\rvert^2\\
			\leq&\int_M2d_5H^{\alpha+\sigma-1}\eta\lvert\nabla H\rvert\lvert\nabla\eta\rvert.
		\end{aligned}
	\end{equation*}
	By the Schwarz inequality, we have
	\begin{equation}\label{4.11}
		\begin{aligned}
			&\int_M\Bigl(2d_1H^{\alpha+\sigma}\eta^2+\sigma d_4 H^{\alpha+\sigma-2}\eta^2\lvert\nabla H\rvert^2\Bigr)\\
			&\leq\int_M\Bigl(\frac{d_3+d_5}{\epsilon}H^{\alpha+\sigma-2}\eta^2\lvert\nabla H\rvert^2+\epsilon d_5H^{\alpha+\sigma}\lvert\nabla\eta\rvert^2+\epsilon d_3H^{\alpha+\sigma}\eta^2\Bigr).
		\end{aligned}
	\end{equation}
	Choosing $\epsilon$ and $\sigma$ such that
	\[\sigma d_4\geq\frac{d_3+d_5}{\epsilon},\ \epsilon d_3\leq d_1.\]
	Then (\ref{4.11}) becomes
	\begin{equation*}
		\int_Md_1H^{\alpha+\sigma}\eta^2\leq\int_M\epsilon d_5H^{\alpha+\sigma}\lvert\nabla\eta\rvert^2.
	\end{equation*}
	For any $i\in \mathbb{N}$, we choose a cut-off function $\eta_i\in C^{\infty}_0(B_o(i+1))$ such that
	\[0\leq\eta_i\leq 1,\ ,\eta_i=1\ in\ B_o(i),\ \lvert\nabla\eta_i\rvert^2\leq10.\]
	Then we have
	\begin{equation*}
		d_1\int_{B_o(i)}H^{\alpha+\sigma}\leq10\epsilon d_5\int_{B_o(i+1)}H^{\alpha+\sigma}.
	\end{equation*}
	The iteration process leads to
	\begin{equation*}
		(\frac{d_6}{\epsilon})^{i-1}\int_{B_o(1)}H^{\alpha+\sigma}\leq \int_{B_o(i)}H^{\alpha+\sigma}.
	\end{equation*}
	Since $H$ is upper bounded, we can infer that
	\begin{equation*}
		Vol(B_o(i))\geq (\frac{d_7}{\epsilon})^{i-1},
	\end{equation*}
	for any $\epsilon$. This contradicts the volume comparison theorem, since we can choose $\sigma$ sufficient large to make $\epsilon$ sufficient small.
\end{proof}

\bibliography{bib}
\bibliographystyle{plain}

Yuxin Dong

School of Mathematical Science \& Laboratory of Mathematics for Nonlinear Science, Fudan University, Shanghai, 200433, PR China

\itshape{Email address}: yxdong@fudan.edu.cn\\

Hezi Lin

School of Mathematics and Statistics \& Key Laboratory of Analytical Mathematics and Applications (Ministry of Education) \& FJKLAMA, Fujian Normal University, Fuzhou, 350117, PR China

\itshape{Email address}: lhz1@fjnu.edu.cn\\

Weihao Zheng

School of Mathematical Science, Fudan University, Shanghai, 200433, PR China

\itshape{Email address}: whzheng20@fudan.edu.cn
\end{document}